\documentclass[a4paper, 11pt, parskip=half]{amsart}

\usepackage{customtemplate}
\usepackage{mathtools}
\usepackage{tabularx}
\usepackage{framed}
\usepackage{array}
\usepackage{shortcuts}
\usepackage{stmaryrd}
\usepackage{bbm}
\begin{document}
	\title{Effective Joint Sato--Tate Distribution and Sign Change of Symmetric Power Coefficients}
	\date{}
	\author{Arvind Kumar, Moni Kumari, Prabhat Kumar Mishra}
    \address{Department of Mathematics, Indian Institute of Technology Jammu, Jagti, PO Nagrota, NH-44 Jammu 181221, J \& K, India\vspace*{-5pt}}
    \email{arvind.kumar@iitjammu.ac.in\vspace*{-6pt}}
	\email{moni.kumari@iitjammu.ac.in\vspace*{-6pt}}
	\email{2022rma0025@iitjammu.ac.in\vspace*{-6pt}}

    \subjclass[2020]{Primary: 11F11; Secondary: 11F30}
	\keywords{Equidistribution, Fourier coefficients, symmetric power $L$-functions, Sign change}

\begin{abstract}
We prove an unconditional, effective joint Sato--Tate distribution for the Fourier coefficients of two twist-inequivalent, non-CM newforms $f$ and $f'$. Our result generalises a result of Thorner, which holds for rectangular regions, by extending it to any measurable region of $[-2,2]^2$ whose boundary consists of finitely many continuous curves of finite length.

As a consequence, we develop a unified framework to study various arithmetic properties of Fourier coefficients of symmetric power $L$-functions attached to $f$ and $f'$. In particular, for these coefficients (and their polynomial expressions), we obtain
effective distribution results, quantitative statements on simultaneous sign behaviour, and bounds for the first sign change.  
\end{abstract}
\maketitle
\section{Introduction and Statements of Main Results}
Let $f(z)=\sum_{n=1}^{\infty}a(n)n^{\frac{k-1}{2}}e^{2 \pi i n  z} \in S_{k}(N)$
be a normalised non-CM newform of weight $k$ and level $N$ with trivial nebentypus. By Deligne’s proof of Weil's conjecture, for every prime $p$, the Fourier coefficient satisfies 
$|a(p)| \le 2.$
While Deligne's bound provides a range for $a(p)$, it is natural to study the distribution of these coefficients. The Sato–Tate conjecture, a result now established in a series of landmark papers by Barnet-Lamb, Geraghty, Harris, and Taylor \cite{lamb}, gives an answer to this question. It asserts that the values ${a(p)}$ are equidistributed in the interval $[-2, 2]$ with respect to the Sato--Tate measure
$$d\mu_{\rm ST}=\frac{1}{\pi}\sqrt{1-\frac{u^2}{4}}du.$$ 
More precisely, for any measurable subset $B \subseteq [-2,2]$,
\begin{equation}\label{eq: st} 
\pi_{f,B}(x) := \#\{p \le x :  p\nmid N, a(p) \in B\} \sim \mu_{\rm ST}(B) \pi(x), \quad \text{as } x \to \infty.
\end{equation}
For many arithmetic applications, one also requires an effective estimate for the rate of convergence in \eqref{eq: st}. In a fundamental work, Thorner \cite[Theorem~1.1]{thorner} established such an estimate for intervals, proving that
\begin{equation*}\label{est} 
  \pi_{f,I}(x)= \mu_{\rm ST}(I)\pi(x) +O\left(  \frac{ \pi(x)\log(kN \log x)}{\sqrt{\log x}}\right),  
\end{equation*} 
uniformly for every closed interval $I\subseteq[-2,2]$, where the implied constant is absolute and effectively computable. More recently, Hoey, Iskander, Jin, and Trejos-Su\'arez \cite[Theorem~1.1]{hoey} determined this constant explicitly. Effective Sato--Tate theorems have since become an important ingredient in the study of the arithmetic of Fourier coefficients; see, for instance, \cite[Theorem 1.1]{luca} and \cite{kim}).

We now consider another normalised non-CM newform $f'(z)=\sum_{n=1}^{\infty}a'(n)n^{\frac{k'-1}{2}}e^{2 \pi i nz}\in S_{k'}(N')$ which is twist-inequivalent to $f$, that is, there exists no primitive character $\chi$ that satisfies $f' =f \otimes \chi$.
In this setting, the object of interest is the sequence of pairs $(a(p),a'(p)) \in [-2,2]^2$. 
Wong \cite{wong} proved that these pairs are equidistributed with respect to the product Sato--Tate measure
$$
d\mu_{\rm JST}(u,v)
=
\frac1{\pi^2}
\sqrt{1-\frac{u^2}{4}}
\sqrt{1-\frac{v^2}{4}}
\,du\,dv.
$$
Equivalently, for every measurable set $E\subseteq[-2,2]^2$ whose boundary has Lebesgue measure zero,
\begin{equation*}\label{eq: jst} 
\pi_{f, f', E}(x) := \#\{p \le x : p \nmid NN',\ (a(p), a'(p)) \in E\} \sim \mu_{\rm JST}(E)\pi(x), \quad \text{as } x \rightarrow \infty.
\end{equation*} 

Obtaining an effective version of this joint distribution is substantially more delicate. Unlike the one-dimensional setting, one needs to approximate the indicator function of a planar region by suitable test functions that depend strongly on the geometry of the region under consideration.

Thorner  \cite[Theorem 1.2]{thorner} obtained an effective version of the joint Sato--Tate theorem in the case where the region $E=R \subset[-2,2]^2$  is a closed rectangle with sides parallel to the coordinate axes. His argument exploits the product structure of rectangles, reducing the problem to one-dimensional approximations and allowing the effective Sato--Tate theorem to be applied in each coordinate separately. Chen and Shen \cite[Theorem~1.5]{shen} obtained effective joint Sato--Tate for certain hyperbolic regions arising in the study of elliptic curves. Their method approximates the hyperbolic domain by a union of sufficiently many rectangles and combines Thorner's estimate with a careful analysis of the geometric approximation error. More recently, Thorner \cite[Theorem~2.1(2)]{thorner2} refined the underlying analytic estimates and proved that
\begin{equation}\label{lemma 1.2 analogue}
\pi_{f,f',R}(x)
=
\mu_{\rm JST}(R)\pi(x)
+
O\!\left(
\pi(x)
\left(
\frac{\log(kk'NN'\log x)}
{\sqrt{\log x}}
\right)^{1/2}
\right),
\end{equation}
for any rectangle $R\subseteq[-2,2]^2$ with sides parallel to the coordinate axes, where the implied constant is positive, absolute, and effectively computable. Furthermore, when the levels of the newforms $N,N'$ are squarefree, the decay rate of the error term in \eqref{lemma 1.2 analogue} is even faster, as the exponent $1/2$ can be removed.

The effective joint Sato--Tate theorems described above already provide powerful tools for studying the distribution of pairs of Fourier coefficients. In many arithmetic problems, one is naturally led to count primes for which the pair
$(a(p),a'(p))$ 
satisfies a prescribed condition. Such conditions define subsets of  $[-2,2]^2$, whose geometry varies according to the problem under consideration. This motivates the search for an effective joint Sato--Tate theorem that applies to a large and flexible class of planar regions.

The main result of this paper provides such a theorem. We establish an effective joint Sato--Tate theorem for Borel subsets of $[-2,2]^2$ whose boundaries are finite unions of continuous curves of finite length. The resulting error term is determined solely by natural geometric characteristics of the boundary, such as its total length and the number of its connected components, and is otherwise independent of the particular shape of the region. Consequently, the theorem applies uniformly to a broad class of regions without requiring a separate geometric analysis for each application.

The versatility of the theorem is illustrated in the next section through several applications concerning the Fourier coefficients of symmetric power $L$-functions attached to two distinct newforms. These include quantitative distribution results for polynomial expressions in the coefficients $a(p^m)$ and $a'(p^n)$, effective estimates for simultaneous sign changes, and upper bounds for the least prime satisfying prescribed sign conditions.

A common principle underlies all of these applications. Each arithmetic condition determines a subset of $[-2,2]^2$, reducing the original problem to estimating the number of primes for which
$(a(p),a'(p)) $
belongs to that subset. Our effective joint Sato--Tate theorem then provides the required asymptotic estimate. In this way, a wide variety of arithmetic questions can be treated within a single framework, providing a unified approach to problems that at first sight appear to be unrelated.

The methods developed in this paper have further applications beyond those presented here. For example, in \cite{kkm2}, they are used to establish a new criterion for determining newforms from arithmetic inequalities among their Fourier coefficients.

\subsection*{Main results}
We now state the main results of the paper. Throughout, we assume that
$$
f(z)=\sum_{n=1}^{\infty}a(n)n^{\frac{k-1}{2}}e^{2\pi inz}\in S_k(N)
\quad\text{and}\quad
f'(z)=\sum_{n=1}^{\infty}a'(n)n^{\frac{k'-1}{2}}e^{2\pi inz}\in S_{k'}(N')
$$
are two normalised non-CM newforms. For $x\ge2$, define
\begin{equation}\label{eq:mx}
\mathcal M(x)=
\left(
\frac{\sqrt{\log x}}
{\log(kk'NN'\log x)}
\right)^{1/2}.
\end{equation}
Our first result establishes an effective joint Sato--Tate theorem for Borel subsets of $[-2,2]^2$ whose boundaries are finite unions of continuous curves of finite length.
\begin{theorem}\label{thm:main}
Let $f$ and $f'$ be two twist-inequivalent non-CM newforms. Let $E\subseteq[-2,2]^2$ be a Borel measurable set whose boundary
$$
\partial E=\Gamma_1\cup\cdots\cup\Gamma_\alpha,
$$
is a union of continuous curves of finite length. Then
$$
\pi_{f,f',E}(x)
=
\mu_{\rm JST}(E)\pi(x)
+
O\!\left(
L\alpha
\frac{\pi(x)}
{\mathcal M(x)^{1/3}}
\right),
$$
where $L$ denotes the total length of $\partial E$. The implied constant is positive, absolute, and effectively computable.
\end{theorem}

 For rectangular regions, Thorner proved a sharper estimate \cite[Theorem~1.2]{thorner2}. The strength of Theorem~\ref{thm:main} is that it applies to arbitrary Borel subsets of $[-2,2]^2$ satisfying the above boundary condition. This level of generality is essential for the applications developed later in the paper. Also, the appearance of the parameters $L$ and $\alpha$ reflects the fact that the error term depends only on the geometry of the boundary of the region and not on its particular shape.

Many regions arising in arithmetic applications satisfy an additional regularity property that allows the error term in Theorem~\ref{thm:main} to be improved. We now formulate the required hypothesis.

\begin{hypothesis}\label{hypothesis}
Let $E\subseteq[-2,2]^2$ be a Borel measurable set whose boundary
$$
\partial E=\Gamma_1\cup\cdots\cup\Gamma_\alpha,
$$
is a union of continuous curves of finite length. Assume that each curve $\Gamma_t$ satisfies one of the following conditions.

\begin{enumerate}
\item
$\Gamma_t$ is contained in a vertical line in $[-2,2]^2$; or

\item
there exists a constant $\beta\ge1$ such that, for every $a\in[-2,2]$,
$$
\#
\left\{
\Gamma_t
\cap
\{(u,v):u=a\}\right\}
\le\beta.
$$
\end{enumerate}
\end{hypothesis}

Hypothesis~\ref{hypothesis} bounds the number of intersections of each boundary component with a vertical line, thereby excluding excessive horizontal oscillation. Under this additional assumption, the exponent in the error term improves from $1/3$ to $1/2$.

\begin{theorem}\label{thm:main2}
Let $f$ and $f'$ be two twist-inequivalent non-CM newforms. Let $E\subseteq[-2,2]^2$ satisfy Hypothesis~\ref{hypothesis}. Then
$$
\pi_{f,f',E}(x)
=
\mu_{\rm JST}(E)\pi(x)
+
O\!\left(
L\alpha\beta
\frac{\pi(x)}
{\mathcal M(x)^{1/2}}
\right),
$$
where $L$ denotes the total length of $\partial E$. The implied constant is positive, absolute, and effectively computable.
\end{theorem}

The error terms in Theorems~\ref{thm:main} and~\ref{thm:main2} depend only on three simple geometric characteristics of the boundary of $E$: its total length, the number of boundary components, and, under Hypothesis~\ref{hypothesis}, the maximal number of intersections with a vertical line. For regions of fixed shape, these quantities are absolute constants. For example, rectangles satisfy $L\le16$, $\alpha\le4$, and $\beta=1$, while the hyperbolic regions considered by Chen and Shen \cite[Theorem~1.5]{shen} satisfy $L\le16$, $\alpha\le6$, and $\beta=1$. Consequently, the same asymptotic estimates apply uniformly to arbitrary regions satisfying the stated hypotheses, with the dependence on the geometry made completely explicit.

 In most of our applications, the region $E$ inside $[-2,2]^2$ is defined by finitely many polynomial inequalities. We call such a region a semi-algebraic region of $[-2,2]^2$.  
It is easy to see that such a region satisfies Hypothesis~\ref{hypothesis}. We record this fact in the following lemma, which will also be used later in the paper. Moreover, for such a region $E$, the geometric parameters $L,\alpha$, and $\beta$ appearing in Theorem \ref{thm:main2} can be bounded only in terms of the degrees of the defining polynomials of the region $E$ (see Lemma \ref{lemma: alpha}). 
\begin{lemma}\label{lem:alg}
Let $E\subseteq[-2,2]^2$ be a Borel measurable set with boundary
$$
\partial E=\Gamma_1\cup\cdots\cup\Gamma_\alpha,
$$
where each $\Gamma_t$ is the zero set in $[-2,2]^2$ of a non-zero polynomial
$P_t(u,v)\in \R[u,v]$. Then $E$ satisfies Hypothesis~\ref{hypothesis}. In particular, one may take
$
\beta=\max_{1\le t\le\alpha}\deg P_t.
$ 
\end{lemma}
%Indeed, if $\Gamma_t$ is not contained in a vertical line, then for every $a\in[-2,2]$, the vertical line $u=a$ intersects the algebraic curve $P_t(u,v)=0$ in at most $\deg P_t$ points. Hence $$ \#\bigl(\Gamma_t\cap\{(u,v):u=a\}\bigr)\le \deg P_t, $$ which verifies the second alternative in Hypothesis~\ref{hypothesis}. Therefore, each boundary component satisfies one of the two conditions in Hypothesis~\ref{hypothesis}.

\begin{remark}\label{remark: under GRH}
Assuming the Generalised Riemann Hypothesis for all Rankin--Selberg $L$-functions
$
\mathrm{sym}^{m} f \otimes \mathrm{sym}^{n} f'$, Thorner \cite[Theorem~1.4]{thorner} obtained a stronger effective joint Sato--Tate estimate for rectangular regions. Combining Thorner's estimate with the arguments developed in this paper yields corresponding improvements in Theorems~\ref{thm:main} and~\ref{thm:main2}. More precisely, the error terms become
$$
\frac{L\alpha x^{17/18}\log(kk'NN'x)^{1/9}}
{(\log x)^{2/3}} \quad {\rm and} \quad 
\frac{L\alpha\beta x^{11/12}\log(kk'NN'x)^{1/6}}
{(\log x)^{1/2}},
$$
respectively. Consequently, every application established in this paper admits a corresponding improvement under GRH.
\end{remark}

\textbf{Notation.}
Throughout this article, the letter $p$ denotes a prime number, while $\ell, m,$ and $n$ are used to denote non-negative integers.
For any set $S$, we write $\#S$ for its cardinality, and $\pi(x)=\#\{p\le x\}$.
Given a function $F$ and a positive function $G$, the notation $F = O(G)$ (or equivalently $F \ll G$) means that there exists a constant $c$ such that $|F| \le c\,G$ in the relevant range.
For a Borel measurable subset $E$ of $[-2,2]^2$, we denote its boundary by $\partial E$, its closure by $\overline E$ and its interior by $E^{\circ}$. We assume that $f$ and $f'$ are normalised non-CM newforms, and by GRH, we mean the Generalised Riemann Hypothesis for all Rankin-Selberg $L$-functions attached to $f$ and $f'$. Finally, if $\Gamma$ is a continuous curve in $\bf R^2$ with parametrisation $\gamma:[0,1]\rightarrow \bf R^2$, then the length of $\Gamma$ is denoted by $\ell(\Gamma)$ and is defined as 
\begin{center}
$\ell(\Gamma)= \sup_{\mathcal P}\sum_{i=0}^{n-1}\mid\mid \gamma(t_{i+1})-\gamma(t_i)\mid\mid,
$\end{center}
where the supremum is over all partitions $\mathcal P:=\{t_0=0<t_1<\dots<t_n=1\}$ of $[0,1]$.

\section{Applications: Distributions and sign change of Symmetric power coefficients}\label{sec:application}
Let
$f(z)=\sum_{n=1}^{\infty}a(n)n^{\frac{k-1}{2}}e^{2\pi inz}\in S_k(N)
$
be a normalised non-CM newform, and let $\pi_f$ denote the associated cuspidal automorphic representation of $\mathrm{GL}_2(\mathbb A)$, where $\mathbb A$ is the ring of adeles over $\Q$. For each integer $m\ge0$, let $\mathrm{sym}^m\pi_f$ denote the $m$th symmetric power lift of $\pi_f$. By the recent works of Newton and Thorne \cite{newton1,newton2}, $\mathrm{sym}^m\pi_f$ is known to be a cuspidal automorphic representation of $\mathrm{GL}_{m+1}(\mathbb A)$ for every $m\ge0$. The associated symmetric power $L$-function is given by
$$
L(\mathrm{sym}^m f,s)
=
\prod_p
\prod_{r=0}^{m}
\left(
1-\frac{\alpha_p^{\,m-r}\beta_p^{\,r}}{p^s}
\right)^{-1}
=
\sum_{n=1}^{\infty}
\frac{a_{\mathrm{sym}^m f}(n)}{n^s},
\qquad
\Re(s)>1,
$$
where $\alpha_p$ and $\beta_p$ are the Satake parameters of $f$ at a prime
$p\nmid N$, satisfying
$$
\alpha_p+\beta_p=a(p),
\qquad
\alpha_p\beta_p=1.
$$

A fundamental feature of the symmetric power coefficients is that, at unramified primes, they are given by the Hecke eigenvalues of prime powers. Indeed, the Hecke relations imply that
\begin{equation}\label{eq:hecke}
a_{\mathrm{sym}^m f}(p)
=
a(p^m)
=
U_m\!\left(\frac{a(p)}{2}\right),
\qquad
p\nmid N,
\end{equation}
where $U_m$ denotes the Chebyshev polynomial of the second kind. Recall that
$$
U_m(\cos\theta)
=
\frac{\sin((m+1)\theta)}{\sin\theta},
$$
and that these polynomials satisfy the recurrence
$$
U_{m+1}(u)=2uU_m(u)-U_{m-1}(u),
$$
with initial values $U_0(u)=1$ and $U_1(u)=2u$.

Now let
$
f'(z)
=
\sum_{n=1}^{\infty}
a'(n)n^{\frac{k'-1}{2}}e^{2\pi inz}
\in
S_{k'}(N')
$
be another normalised non-CM newform with associated automorphic representation $\pi_{f'}$. For integers $m,n\ge0$, the Rankin--Selberg $L$-function attached to
$\mathrm{sym}^m\pi_f\otimes\mathrm{sym}^n\pi_{f'}$
is defined by
$$
L(\mathrm{sym}^m f\times\mathrm{sym}^n f',s)
=
\prod_p
\prod_{i=0}^{m}
\prod_{j=0}^{n}
\!\!\left(
1-
\frac{\alpha_p^{m-i}\beta_p^{i}
{\alpha_p'}^{\,n-j}{\beta_p'}^{\,j}}
{p^s}
\right)^{-1}\!\!\!\!\!\!
=\!\!
\sum_{r=1}^{\infty}
\frac{a_{\mathrm{sym}^m f\times\mathrm{sym}^n f'}(r)}
{r^s},
\quad
\Re(s)>1.
$$

The results of Newton and Thorne also imply that
$L(\mathrm{sym}^m f,s)$ extends to an entire function and satisfies the expected functional equation. The same is true for
$L(\mathrm{sym}^m f\times\mathrm{sym}^n f',s)$,
except when
$\mathrm{sym}^m\pi_f\simeq\mathrm{sym}^n\pi_{f'}$,
in which case it has simple poles at $s=0$ and $s=1$.

Assume henceforth that $f$ and $f'$ are twist-inequivalent. Then, for every prime $p\nmid NN'$,
$$
a_{\mathrm{sym}^m f\times\mathrm{sym}^n f'}(p)
=
a_{\mathrm{sym}^m f}(p)\,
a_{\mathrm{sym}^n f'}(p)
=
a(p^m)a'(p^n).
$$
Together with \eqref{eq:hecke}, this shows that the coefficients of the symmetric power Rankin--Selberg $L$-function at primes are polynomial functions of the pair $(a(p),a'(p))$. Consequently, every polynomial expression in the coefficients $a(p^m)$ and $a'(p^n)$ can be written in the form
$
P(a(p),a'(p)),
$
for a suitable polynomial $P\in\R[x,y]$. Since the pairs $(a(p),a'(p))$ are equidistributed in $[-2,2]^2$ with respect to the joint Sato--Tate measure $\mu_{\rm JST}$, it follows that the values
$$
P(a(p),a'(p))
$$
are distributed according to the pushforward of $\mu_{\rm JST}$ under $P$.

The main goal of this section is to make this distribution effective by applying Theorem~\ref{thm:main2}. This yields quantitative distribution results for polynomial expressions in the Fourier coefficients of two newforms, from which we derive several arithmetic consequences for the coefficients of symmetric power Rankin--Selberg $L$-functions.

Our first application of Theorem~\ref{thm:main2} is an effective equidistribution theorem for polynomial expressions in the Fourier coefficients of two newforms.
\begin{theorem}\label{thm:sign change}
Let $f$ and $f'$ be two twist-inequivalent newforms. Let $P(u,v)\in \R[u,v]$ be a non-constant polynomial of degree $\delta$. For any real interval $I$   we have
$$
\#\{p\le x :  P(a(p),a'(p))\in I\}= \mu_{\rm JST}(E_{I,P})\pi(x) + O\left( \delta^8 \dfrac{\pi(x)}{\mathcal M(x)^{1/2}}\right),$$ 
where 
\begin{equation}
    E_{I,P}:=\{(u,v)\in [-2,2]^2: P(u,v)\in I\}
\end{equation}
and $\mathcal{M}(x)$ is defined by \eqref{eq:mx}. Here the implied constant is positive, absolute, and effectively computable,  also, it is %It depends on the degree of the polynomial $P$, and is 
 independent of the interval $I$.
\end{theorem}
We remark that the exponent of $\delta$ in the error term is not optimal. Nevertheless, it shows the dependence of parameters $L,\alpha,$ and $\beta$ appearing in Theorem \ref{thm:main2} only on the degree of the underlying polynomial, and this is sufficient for the results obtained in this paper.

We now derive several arithmetic applications of Theorem~\ref{thm:sign change}. In each case, the arithmetic condition involving the symmetric power coefficients $a(p^m)$ and $a'(p^n)$ can be reformulated as the condition that the pair $(a(p),a'(p))$ lies in a suitable semi-algebraic subset of $[-2,2]^2$. Equivalently, it can be expressed as a polynomial inequality of the form
$$
P(a(p),a'(p))\in I,
$$
for an appropriate polynomial $P$ and interval $I$. This reformulation allows us to apply Theorem~\ref{thm:sign change} directly, yielding desirable effective asymptotic formulas together with explicit estimates for the corresponding densities.

Theorem~\ref{thm:sign change} immediately gives the following estimate, which may be viewed as a variant of the multiplicity one phenomenon.
\begin{corollary}\label{cor:density of zero} Let $f$ and $f'$ be two twist-inequivalent newforms. Let $P(u,v)\in \R[u,v]$ be a non-constant polynomial. For any $m,n\ge 0$ with $(m,n)\neq (0,0)$,
     $$\#\{p\le x : P(a(p^m), a'(p^n))=0\} = O\left( \dfrac{\pi(x)}{\mathcal M(x)^{1/2}}\right).
$$
\end{corollary}
The proof follows by applying Theorem~\ref{thm:sign change} for the polynomial $P(U_m(u/2),U_n(v/2))$ and $I=\{0\}$ and using the fact that the zero set of a non-constant polynomial has Lebesgue measure zero (cf.~\cite[Remark~2.18]{ak}).

The above result recovers and generalises some of the classical multiplicity one results. Indeed, taking $P(u,v) = u-v$, and  $m=n=1$ yields \cite[Theorem 1]{murty17}, \cite[Theorem 2]{rajan17}, while the choice $m=n=2$  recovers \cite[Corollary of Theorem A]{ramakrishnan00}. 

We emphasize that the implied constant in Corollary~\ref{cor:density of zero} as well as in all subsequent results of this section, is positive, absolute, effectively computable, and it depends only on the integers $m,n$, and the degree of the polynomial $P(u,v)$.
\subsection{Simultaneous sign change of symmetric power coefficients}
In this subsection, we determine effective density results of primes for which the coefficients of two symmetric power $L$-functions, as well as their polynomial expressions, are simultaneously positive or simultaneously negative.  
To begin with, we state the following result, which provides an effective and explicit estimate for how often the symmetric power coefficients $a(p^m)$ and $a'(p^n)$ have the same or opposite signs.
\begin{theorem}\label{thm:simultaneous sign change of symmetric power coefficients}
Let $f$ and $f'$ be two twist-inequivalent newforms.  For integers $m,n\ge 0$ with $(m,n)\neq (0,0)$, we have
\begin{enumerate}
    \item 
$ \#\{p\le x :  a(p^m)a'(p^n)>0\}= d_{m,n}{\pi(x)} + O\left( \dfrac{\pi(x)}{\mathcal M(x)^{1/2}}\right)$,
        \item 
        $ \#\{p\le x :  a(p^m)a'(p^n)<0\}= (1-d_{m,n})\pi(x) + O\left( \dfrac{\pi(x)}{\mathcal M(x)^{1/2}}\right)$,
        \end{enumerate}
        where
  \begin{equation}\label{eq:dmn}
    \displaystyle{d_{m,n} = \begin{cases}
                    \frac{1}{2}, & \text{if } m \text{ or } n \text{ is odd},\\
                    \frac{1}{2}+\frac{1}{2\pi^2} 
                    \left( \tan\left(\frac{\pi}{m+1}\right) -\frac{\pi}{m+1} \right) \left(\tan\left(\frac{\pi}{n+1}\right) -\frac{\pi}{n+1} \right) , & \text{otherwise}.                    
                \end{cases}}
  \end{equation}
\end{theorem}
Theorem~\ref{thm:simultaneous sign change of symmetric power coefficients}
recovers and extends several earlier results on sign changes of symmetric power
coefficients including
\cite[Theorem~1.5]{kumari} when $m=n=1$,
\cite[Theorem~1.1]{jaban} when $n=0$, and
\cite[Theorem~1.3]{amri} when $m=n$ is odd. In each case, our result provides an explicit and effective error term.
\begin{remark}\label{remark:positive}
If both $m$ and $n$ are even, then
$ d_{m,n} > \frac{1}{2},$ 
as $\frac{\pi}{\ell+1} < \tan\!\left(\frac{\pi}{\ell+1}\right)$ for  any even integer $\ell \ge 2$. Thus, in this situation, the coefficients $ a(p^m)a'(p^n)$ are positive more often than negative.
Furthermore, the Taylor expansion of $\tan x$ around $x=0$
gives
$$
d_{m,n}
= \frac{1}{2}
+ \frac{\pi^4}{18}\frac{1}{(m+1)^3 (n+1)^3}
+ o\!\left((m+1)^{-3}(n+1)^{-3}\right)
$$
which shows that $d_{m,n}\rightarrow \frac{1}{2}$ as $m,n \rightarrow \infty$ through even integers.
 \end{remark}   

Theorem~\ref{thm:simultaneous sign change of symmetric power coefficients}
determines the asymptotic distribution of the sign of the product
$a(p^m)a'(p^n)$, corresponding to the polynomial
$P(u,v)=uv.$
The next theorem identifies a general symmetry principle governing the sign distribution of polynomial expressions in symmetric power coefficients. It shows that, whenever $P(u,v)$ satisfies one of the symmetry conditions below, the sets of primes for which
$P(a(p^m),a'(p^n))$ is positive and negative both have natural density $1/2$, with an effective error term. 
\begin{theorem}\label{cor:simultaneous}
Let $f$ and $f'$ be two twist-inequivalent newforms.  Let $P(u,v)\in \R[u,v]$ be a non-zero polynomial, and $m,n$ be non-negative integers with $(m,n)\neq (0,0)$. Further, assume that either of the following two conditions holds.
\begin{enumerate}
    \item 
   Either $m$ or $n$ is odd and 
   $
   P(su,tv)=-P(u,v), {{~for~ some~}} s\in\{(\pm 1)^m\},\, t\in\{(\pm 1)^n\}.
   $
   \item 
   $m=n$ and 
   $
   P(u,v)=-P(v,u).
   $
\end{enumerate}
Then, we have   
$$
\#\{p\le x : P(a(p^m),a'(p^n))\lessgtr 0\}= \frac{\pi(x)}{2} + O\left( \dfrac{\pi(x)}{\mathcal M(x)^{1/2}}\right).$$ 
\end{theorem}
The symmetry assumptions in Theorem~\ref{cor:simultaneous} are satisfied by many natural families of polynomials. For example,
\begin{itemize}
    \item If $m$ is odd, then 
    $P(u,v)=uQ(u^2,v)$ for any polynomial  $Q(u,v)$.
   \item If $n$ is odd, then 
  $P(u,v)=vQ(u,v^2)$ for any polynomial  $Q(u,v)$.
  \item If $m\!=\!n$, then  $P(u,v)\!=\!(u-v)Q(u,v)$ for any symmetric polynomial $Q(u,v)$.
\end{itemize}

An immediate consequence of Theorem~\ref{cor:simultaneous} is an effective comparison between symmetric power coefficients. Indeed, taking
$P(u,v)=u-v$
gives the following result, which estimates the number of dominating coefficients.
\begin{cor}\label{thm:compare}
Let $f$ and $f'$ be two twist-inequivalent newforms.  Let $m,n$ be non-negative integers such that $(m,n)\neq (0,0)$. If $m=n$ or both are odd, then we have
$$
\#\{p\le x: a(p^m) < a'(p^n)\}= \frac{\pi(x)}{2} + O\left(\frac{\pi(x)}{\mathcal M(x)^{1/2}}\right).$$
\end{cor} 
Corollary~\ref{thm:compare} extends \cite[Proposition~2.1]{chiriac2} from Fourier coefficients of newforms to symmetric power coefficients, while also providing an explicit effective error term.

\subsection{First simultaneous sign change}
Theorem~\ref{thm:simultaneous sign change of symmetric power coefficients}
shows that the coefficients $a(p^m)$ and $a'(p^n)$ have opposite signs for a positive proportion of primes. A natural quantitative question is how soon such a prime must occur. In this subsection, we obtain an unconditional upper bound for the least prime with this property.

For a normalised non-CM newform $f\in S_k(N)$ and an integer $m\ge1$, define
$$
n_f:=\min\{n\ge1:a(n)<0\},
\qquad
p_{f,m}:=\min\{p:a(p^m)<0\}.
$$
The problem of finding a suitable upper bound for $n_f$ and $p_{f,m}$  has received considerable attention. Using convexity bounds for automorphic $L$-functions, Iwaniec, Kohnen, and Sengupta \cite{Iwaniec-Kohnen-Sengupta} proved that
$$
n_f\ll (k^2N)^{29/60}.
$$
This exponent was subsequently improved to $9/20$ by Kowalski, Lau, Soundararajan, and Wu \cite[Theorem~1]{kowalski-lau-soundararajan-wu}, and later to $3/8$ by Matom\"aki \cite[Theorem~1]{matomaki1}. In each case, the implied constant is absolute.

For prime power coefficients, Kowalski, Lau, Soundararajan, and Wu
\cite[Theorem~2]{kowalski-lau-soundararajan-wu} proved that, if the level $N$ is squarefree, then
$$
p_{f,m}\ll\log(kN),
$$
except possibly for a family of newforms of cardinality
$$
\ll
(kN)
\exp\!\left(
-\frac{c\log(kN)}{\log\log(kN)}
\right),
$$
where $c>0$ is an absolute constant. A conditional bound for $p_{f,1}$ was later obtained by Xu \cite[Theorem~1.2]{zhao}.

We now consider the simultaneous analogue of this problem. Let
$f'\in S_{k'}(N')$ be another non-CM newform, and let
$m,n\ge0$ with $(m,n)\neq(0,0)$. Define
$$
p_{f,f',m,n}
:=
\min\{p:a(p^m)a'(p^n)<0\}.
$$
Our next theorem gives the first unconditional upper bound for
$p_{f,f',m,n}$.

\begin{theorem}\label{thm:first sign change}
Let $f$ and $f'$ be two twist-inequivalent newforms. Then there exists a constant $C>0$ such that
$$
p_{f,f',m,n}
\ll
\exp\!\left(
C(\log kk'NN')^2
\right).
$$
The implied constant is absolute and depends only on $m$ and $n$.
\end{theorem}

Several authors (for example, \cite{gun3,hua}) have studied the least integer $\nu$ such that 
$a(\nu^m)a'(\nu^n)<0$. To the best of our knowledge, however, Theorem~\ref{thm:first sign change} is the first result giving an upper bound for the least \emph{prime}
$p$ such that
$a(p^m)a'(p^n)<0$.

Although the bound in Theorem~\ref{thm:first sign change} is unlikely to be optimal, it is completely unconditional. Under the Generalised Riemann Hypothesis, combining the proof of Theorem~\ref{thm:first sign change} with the effective joint Sato--Tate estimate stated in Remark~\ref{remark: under GRH} yields
$$
p_{f,f',m,n}
\ll
(\log(kk'NN'))^{2+\varepsilon},
\qquad
\text{for every }\varepsilon>0,
$$
where the implied constant is absolute and depends only on
$m$, $n$, and $\varepsilon$.

\section{Proof of \Cref{thm:main}}
We begin the proof by approximating the region $E$ by a grid of small rectangular boxes. For that, 
let $m$ be a large positive integer (to be chosen later) and $\tau$ be a transcendental number such that $1< \tau < 1+\frac{1}{2m}$. 
Define horizontal and vertical line segments 
$$
 h_i: \left\{\left(u, -2 +  \frac{4\tau i}{m}\right): -2\le u\le 2\right\}  , \quad  v_i: \left\{\left(-2 +\frac{4\tau i}{m},v\right):-2\le v\le 2\right\}, \quad 0\le i\le m-1,
$$ 
 respectively together with 
 $$h_m: \{(u, 2): -2\le u\le 2\} , \quad v_m:\{(2, v): -2\le v\le 2\}.$$
 These $m+1$ horizontal and $m+1$ vertical lines  partition the square $[-2,2]^2$ into $m^2$ closed  boxes $B_{ij}$, where $B_{ij}$ is the region bounded by the lines $h_i, h_{i+1}, v_j$ and $v_{j+1}$, for $0\le i,j\le m-1$. 

Observe that for any prime $p$, the pair $(a(p),a'(p))$ lies in a unique box $B_{ij}$. To see this, suppose $(a(p),a'(p))$ lies in two different boxes. Then it necessarily lies on one of the lines $h_i$ or $v_i$, for some $1\le i\le m-1$. This, however, is not possible because $a(p)$ and $a'(p)$ are algebraic integers, whereas on each of the lines $h_i$ and $v_i$, at least one of the coordinates is transcendental.
We now define the sets
$$
         \mathcal{S} = \bigcup\limits_{B_{ij}\subset E^{\circ}} B_{ij}  
         \quad {\rm and}\quad 
          \mathcal{T} = \bigcup\limits_{B_{ij}\cap \partial {E}\neq \varnothing} B_{ij}. 
$$

Let $(a(p), a'(p))\in E$ and $B_{ij}$ be the unique box containing that pair. We claim that
\begin{enumerate}
    \item 
    $(a(p),a'(p))\in \mathcal{S}\cup (E\cap \mathcal{T})$, and
    \item 
    $ (a(p),a'(p))\not\in \mathcal{S}\cap (E\cap \mathcal{T})$.
\end{enumerate}

To prove the first part, let $(a(p),a'(p))\not \in\mathcal{S}\cup (E\cap \mathcal{T})$. Then $B_{ij}\not\subset E^{\circ}$ and $B_{ij}\cap\partial E= \varnothing$. Since $B_{ij}\cap E\neq \varnothing$, we obtain that  
$$
        B_{ij}  =  (B_{ij}\cap E^{\circ}) \cup (B_{ij}\cap(\overline{E})^c),
$$
giving a disjoint union of non-empty open sets of $B_{ij}$, which is not possible because $B_{ij}$ is connected.
To prove the second part, let $(a(p),a'(p))\in \mathcal{S}\cap (E\cap \mathcal{T})$. Then 
$B_{ij}\subset E^{\circ}$ and $B_{ij}\cap\partial E\neq \varnothing$, simultaneously, which is absurd.

Thus, we have
\begin{equation}\label{first step}
    \pi_{f,f',E }(x) =  \pi_{f,f',\mathcal{S} }(x) +  \pi_{f,f',E \cap \mathcal{T} }(x).
\end{equation}
We first analyse the term $\pi_{f,f',E \cap \mathcal{T} }(x)$. 
Since $\pi_{f,f',E \cap \mathcal{T} }(x) \le \, \pi_{f,f', \mathcal{T} }(x) $,  applying the effective Sato--Tate estimate \eqref{lemma 1.2 analogue} to each of boxes gives
\begin{equation*}
    \pi_{f,f',E \cap \mathcal{T} }(x) 
     \le  \sum\limits_{B_{ij}\subset \mathcal{T}}\left( \mu_{\rm JST}(B_{ij})\pi(x) +  O\left(\frac{\pi(x)}{\mathcal{M}(x)} \right) \right).
\end{equation*}
As each box $B_{ij}$ has Lebesgue measure at most  $\frac{16\tau^2}{m^2}$, it follows that $\mu_{\rm JST}(B_{ij}) \le \frac{16 \tau^2}{m^2}$ and hence 
\begin{equation}\label{eq:estimation of T}
    \pi_{f,f',E \cap \mathcal{T} }(x) 
     \le  \left(\sum\limits_{B_{ij}\subset \mathcal{T}} 1 \right)\left( \frac{16 \tau^2}{m^2} \pi(x)+ O\left(\frac{\pi(x)}{\mathcal{M}(x)}\right)\right).
\end{equation}
It therefore suffices to bound the number of boxes intersecting the boundary $\partial E$. We claim that 
\begin{equation}\label{eq:N1}
    \mathcal N:= \sum\limits_{B_{ij}\subset \mathcal{T}} 1=O(Lm\alpha).
\end{equation}
To prove this, we note that each box $B_{ij}$ has side lengths at most
$\frac{4\tau}{m},$ and hence diameter at most 
\begin{equation}\label{eq:delta}
\delta= \frac{4\sqrt{2} \tau}{m}.
\end{equation}
If $B_{ij}\subset \mathcal T$, then  $B_{ij}\cap\partial E\neq\varnothing$, and therefore $B_{ij}$ is contained in the
$\delta$-neighborhood of $\partial E$
$$
(\partial E)_{\delta}
:=\{(u,v)\in\R^2:{\rm dist}((u,v),\partial E)\le\delta\}.
$$
Consequently, the union of all boxes contained in $\mathcal T$ is also contained in
$(\partial E)_{\delta}$, and therefore
$$
\mathrm{Area}(\mathcal T) \le \mathrm{Area}((\partial E)_{\delta}).
$$ 
Moreover, each box has side length at least $2/m$ and so the area of $\mathcal T$ is at least $4 \mathcal N/m^2$, giving
\begin{equation}\label{eq:N}
    \mathcal N \le \frac{m^2}{4}\mathrm{Area}((\partial E)_{\delta}).
\end{equation}
Note that $\partial E = \Gamma_1\cup\cdots \cup \Gamma_{\alpha}$ and $\Gamma_t$ is a curve of length $L_t$ (say) for each $1\le t\le \alpha$.
For a fixed $t$, choose an integer $r$ such that 
$$\frac{L_t}{\delta} \le r <  \frac{L_t}{\delta}+1.$$
Let $\gamma_t:[0,1]\rightarrow \R^2$ be a parametrization of $\Gamma_t$. Then there exists a partition $\{u_0=0<u_1<\dots<u_{r-1} <u_r=1\}$ of $[0,1]$  such that the length of the arcs joining $\gamma_t(u_{i-1})$ and $\gamma_t(u_i)$ equals $\frac{L_t}{r}$. Then it is easy to see that
$$
        (\Gamma_t)_{\delta} \subset \bigcup\limits_{i=0}^{r} B_{2\delta}(\gamma_t(u_i)),
$$
where $B_{2\delta}(\gamma_t(u_i))$  is  a closed ball of radius $2\delta$ centered at $\gamma_t(u_i)$. Thus, it follows that 
$$
\mathrm{Area}((\Gamma_t)_{\delta}) \le 4\pi(r+1)\delta^2 \le 4\pi L_t \delta +8\pi \delta^2.
$$
Summing over $1\le t\le \alpha$, we obtain
$$
\mathrm{Area}((\partial E)_{\delta}) \le  4\pi L \delta  + 8\pi \alpha \delta^2.
$$
Using this in \eqref{eq:N} and substituting $\delta=\frac{4\sqrt 2 \tau}{m}$ give $\mathcal N=O( Lm \alpha)$ proving \eqref{eq:N1}.

Therefore, from \eqref{eq:estimation of T} it follows that
\begin{equation}\label{second term estimation}
\pi_{f,f',E \cap \mathcal{T} }(x) 
     = O\left( \frac{ L \alpha }{m}\pi(x) +  L  m\alpha  \frac{\pi(x)  }{\mathcal{M}(x)}\right).
    \end{equation}
It remains to estimate $\pi_{f,f', \mathcal{S}}(x)$ appearing in $\eqref{first step}$, which we analyse now.
From the definition of $\mathcal{S}$, we write
$$
        \pi_{f,f',\mathcal{S}}(x) = \sum\limits_{B_{ij}\subset E^{\circ}} \pi_{f,f',B_{ij}}(x).
$$
Applying the effective Sato--Tate estimate \eqref{lemma 1.2 analogue} to each box and noting that $\mathcal S$ contains at most $m^2$ boxes, we obtain
\begin{equation}\label{eq: first term estimation part(i)}
      \pi_{f,f',\mathcal{S}}(x) =\mu_{\rm JST}(\mathcal{S}) \pi(x) +  O\left(   m^2\frac{\pi(x)}{\mathcal{M}(x)}\right).  
\end{equation}
Putting \eqref{second term estimation} and \eqref{eq: first term estimation part(i)}   in \eqref{first step}, 
gives
$$
        \pi_{f,f',E }(x) = \mu_{\rm  JST}(\mathcal{S})\pi(x) + O\left( m^2 \frac{\pi(x)}{\mathcal{M}(x)}\right) +  O\left( \frac{ L\alpha}{m}\pi(x)\right) +  O\left(  L m \alpha \frac{\pi(x)}{\mathcal{M}(x)}\right).
$$
Choosing $$m = \left\lfloor \mathcal{M}(x)^{1/3}\right\rfloor + 1,$$ 
where $\lfloor \cdot \rfloor$ is the standard floor function, we obtain
$$
        \pi_{f,f',E }(x) =\mu_{\rm JST}(\mathcal S) \pi(x) +  O\left( L\alpha\frac{ \pi(x)    }{\mathcal{M}(x)^{1/3}}\right).
$$
Note that $m\rightarrow \infty$ as $x\rightarrow \infty$, and so from \eqref{eq:N1} and \eqref{eq:delta}, we obtain that $\mu_{\rm JST}(\mathcal T)\rightarrow 0$ as $x \rightarrow \infty$. Furthermore, because
$\mathcal S\subset E^{\circ}\subset \mathcal S \cup \mathcal T $,  therefore
$$\mu_{\rm JST}(\mathcal S)\rightarrow \mu_{\rm JST}(E),\quad  {\rm as}~ x \rightarrow \infty. $$
Hence, for any sufficiently large $x$
$$
        \pi_{f,f',E }(x) =\mu_{\rm JST}(E) \pi(x) +  O\left( L\alpha\frac{ \pi(x)    }{\mathcal{M}(x)^{1/3}}\right).
$$
This completes the proof of Theorem \ref{thm:main}.
\section{Proof of \Cref{thm:main2}}
We follow the notation of the proof of Theorem \ref{thm:main}. As before, for any fixed $m$, the lines $h_i$ and $v_i$, for $0\le i\le m$, divide $[-2,2]^2$ into $m^2$ closed boxes  and lead to decomposition of $E$ as $E = \mathcal{S}\cup (E\cap \mathcal{T}) $ and hence we can write
\begin{equation}\label{first step 0}
    \pi_{f,f',E }(x) =  \pi_{f,f',\mathcal{S} }(x) +  \pi_{f,f',E \cap \mathcal{T} }(x).
\end{equation}
Recall that the estimate for $\pi_{f,f',E \cap \mathcal{T} }(x)$ has already been  computed in \eqref{second term estimation} which gives
\begin{equation}\label{second term estimation a}
\pi_{f,f',E \cap \mathcal{T} }(x) 
     = O\left( \frac{ L\alpha }{m}\pi(x) +  L m \alpha \frac{\pi(x) }{\mathcal{M}(x)}\right).
    \end{equation}
    It remains to estimate $\pi_{f,f',\mathcal{S} }(x)$, which has an improvement coming from the geometric restrictions in Hypothesis \ref{hypothesis}. 

Fix $j$ such that $0\le j\le m-1$, and consider the $j$th vertical strip 
 \begin{align*}
V_j=\begin{cases}
    \Bigl[-2+\frac{4\tau j}{m}, -2+\frac{4\tau (j+1)}{m}\Bigr]\times[-2,2], \quad &{\rm if} ~j\le m-2, \\
    ~\\
\Bigl[-2+\frac{4\tau (m-1)}{m}, 2 \Bigr]\times[-2,2], \quad & {\rm if} ~j=m-1.
\end{cases}
 \end{align*}
We note that after merging adjacent boxes, $\mathcal S \cap V_j$ can be written as a union of connected components, which are rectangular vertical substrips. Assuming that the number of such components is $N_j$, we claim that 
$$
N_j\le 1+\alpha\beta.
$$
To prove this, recall that $\partial E = \Gamma_1\cup \cdots \cup \Gamma_{\alpha}$, where each $\Gamma_t$ is either a vertical line or it intersects any vertical line at most $\beta$ times.
Choose
$$
a_j\in\Bigl(-2+\frac{4\tau j}{m}, -2+\frac{4\tau (j+1)}{m}\Bigr) \quad {\rm if}~j\le m-2 \quad {\rm and} \quad a_j\in\Bigl(-2+\frac{4\tau j}{m}, 2\Bigr) \quad {\rm if}~j= m-1
$$
such that for any $t$, the curve $\Gamma_t$ does not lie entirely on the vertical line
$
L_j:=\{(u,v):u=a_j\}.
$
Since each intersection with $L_j$ increases the number of connected components by at most one, it is clear that
$$
N_j \le 1+\#(\partial E\cap L_j)= 1+\sum_{t=1}^\alpha \#(\Gamma_t\cap L_j)\le 1+\alpha\beta.
$$
As there are $m$ vertical strips,  after merging adjacent boxes in each vertical strip, the set $\mathcal S$ is a union of at most $m(1+\alpha\beta)$ connected components that are vertical substrips such that any two such substrips in $V_j$ are disjoint. It follows that each of these vertical substrips of $\mathcal S$ is a rectangle, and any two of them can intersect only on the lines $v_j$ for some $1\le j\le m-1$ whose $u$-coordinates are transcendental. Thus, for any prime $p$, the pair lies in a unique such vertical substrip of $\mathcal S$. Applying the estimate \eqref{lemma 1.2 analogue} to each of these substrips and summing over all of them, we obtain
\begin{equation}\label{first term estimation b}
       \pi_{f,f',\mathcal{S} }(x) =  \mu_{\rm JST}(\mathcal{S})\pi(x) + O\left(  m\alpha\beta \frac{ \pi(x)}{\mathcal{M}(x)} \right).
\end{equation}
Substituting \eqref{second term estimation a} and \eqref{first term estimation b} into \eqref{first step 0} yields
$$
        \pi_{f,f',E }(x) =\mu_{\rm JST}(\mathcal{S}) \pi(x) + O\left(m\alpha\beta\frac{ \pi(x) }{\mathcal{M}(x)} \right) +  O\left( \frac{ L\alpha}{m}\pi(x) + L m \alpha \frac{\pi(x)}{\mathcal{M}(x)}\right).
$$
Choosing
$$m = \left\lfloor \mathcal{M}(x)^{1/2} \right\rfloor + 1,$$
we obtain
  $$\pi_{f,f',E }(x) = \mu_{\rm JST}(\mathcal S)\pi(x) +  O\left( L\alpha\beta \frac{\pi(x)  }{\mathcal{M}( x)^{1/2}}\right).$$
As argued in the proof of \Cref{thm:main}, we have  $\mu_{\rm JST}(\mathcal S)\rightarrow \mu_{\rm JST}(E)$ as $x \rightarrow \infty$.
Thus, for sufficiently large $x$
$$
        \pi_{f,f',E }(x) = \mu_{\rm JST}(E)\pi(x) +  O\left( L\alpha\beta \frac{\pi(x)  }{\mathcal{M}( x)^{1/2}}\right),
$$
which completes the proof.

\section{Proof of \Cref{thm:sign change}}
We first establish an auxiliary lemma.
\begin{lemma}\label{lemma: alpha}
Let $P(u,v)\in {\R}[u,v]$ be a non-constant polynomial of degree $\delta$, and $$F  =\{(u,v)\in\R^2:P(u,v)=0\}.$$ Then the number of connected components of $
F\cap[-2,2]^2$ is finite. Further, if 
$$
F\cap[-2,2]^2=\Gamma_1\cup \cdots\cup \Gamma_{\alpha},
$$
is the decomposition into connected components,
then the following assertions hold.
\begin{enumerate}
\item 
Each $\Gamma_i$ is continuous.
\item
$
\alpha\ll\delta^3.$
\item 
$\operatorname{length}(\Gamma_i)\le 2\sqrt 2 \pi \delta$, and hence $\sum_{i=1}^{\alpha}\operatorname{length}(\Gamma_i)\ll\delta^4.$
\end{enumerate}
Here, the implied constants are absolute and effectively computable.
\end{lemma}
\begin{proof}
It suffices to consider the case when $P$ is irreducible, since the general case follows by applying the argument below to each irreducible factor and summing the resulting bounds.
%We first argue that $F\cap[-2,2]^2$ has finitely many connected components. 
Since
$$
F\cap[-2,2]^2
=
\{(u,v)\in[-2,2]^2:P(u,v)=0\},
$$
is a semi-algebraic subset of $\R^2$, therefore by
\cite[Theorem 2.4.5]{bochnak}, it has finitely many connected components. Thus, we can write
$$
F\cap[-2,2]^2=\Gamma_1\cup\cdots\cup \Gamma_{\alpha},
$$
as the union of connected components.\\
\textbf{$(i)$} We notice that for each $1\le i\le \alpha$, the connected component $\Gamma_i$  is a compact, connected, locally
connected metric space. Hence, by \cite[Theorem 3-30]{hocking}, each component is a continuous image of a closed interval, and therefore a
continuous curve.

\textbf{$(ii)$} The claim is immediate when $\delta=1$. Assume now that $\delta\ge 2$.
Let
$$
        F = \Gamma_1 ' \cup \cdots\cup \Gamma_\eta ',
$$
be the decomposition of $F$ into its connected components.  By \cite[Theorem 2]{milnor}, we have
 $\eta\le \delta (2\delta-1).$
 To obtain an upper bound of $\alpha$, we compute the number of connected components of $\Gamma_i '\cap [-2,2]^2,$ for $1\le i \le \eta$, for that, let
$$
b_0(\Gamma_i')
:=
\#\bigl\{
\text{connected components of }
\Gamma_i'\cap[-2,2]^2
\bigr\}, \quad \text{for } {1\le i\le \eta}.
$$
Since $\Gamma_1,\ldots,\Gamma_\alpha$ are the connected components of $F\cap[-2,2]^2$ , we obtain
$$
\alpha
=
\sum_{i=1}^{\eta} b_0(\Gamma_i').
$$
Let $M_i$ denote the number of intersection points of $\Gamma_i'$ with the boundary of the square $[-2,2]^2$. As the boundary comprises of four lines
$
u=\pm2,
v=\pm2,
$
B\'ezout's theorem implies that
$$
M_i\le 4\delta.
$$
Furthermore, each intersection point of $\Gamma_i '$ with the boundary lines increases the number of connected components of $\Gamma_i '\cap[-2,2]^2$ at most by $1$, thus   
$$
b_0(\Gamma_i')
\le M_i+1
\le 4\delta+1.
$$
Summing over $i$, we obtain
$$
\alpha
\le
\sum_{i=1}^{\eta}(4\delta+1)
=
(4\delta+1)\eta
\le
(4\delta+1)\delta(2\delta-1)\ll\delta^3.
$$
\textbf{$(iii)$}
 By the Cauchy--Crofton formula (see \cite[Theorem 5]{ayari}), for each $1\le i\le \alpha$,
\begin{equation}\label{eq:cauchy-crofton}
{\rm length}(\Gamma_i)
=
\frac14
\int_0^{2\pi}
\int_{\R}
N_i(r,\psi)\,dr\,d\psi,
\end{equation}
where
\begin{equation}\label{eq: N i r psi}
        N_i(r,\psi)
            \coloneqq\#\Big\{\Gamma_i \cap\{(u,v)\in[-2,2]^2:u\cos\psi+v\sin\psi=r\}\Big\}.
\end{equation}
If $(u,v)\in [-2,2]^2$, then $
            |u\cos\psi+v\sin\psi| \le 2(|\cos\psi|+|\sin\psi|) \le 2\sqrt2,
$
which gives $N_i(r,\psi)=0$ whenever $|r|>2\sqrt2$. Therefore, \eqref{eq:cauchy-crofton} becomes
\begin{equation}\label{eq:crofton-restricted}
{\rm length}(\Gamma_i)
=
\frac14 \int_0^{2\pi} \int_{-2\sqrt2}^{2\sqrt2}    N_i(r,\psi)\,dr\,d\psi.
\end{equation}
 Fix $(r,\psi)$ and consider the line
$$
    L(r,\psi) = \{(u,v)\in\R^2 : u \cos\psi + v \sin\psi = r\}.
$$
We claim that the intersection number $N_i(r,\psi)$ is at most $\delta$, unless $L(r,\psi)\cap [-2,2]^2\subset \Gamma_i$. 
In order to prove this, we see that restricting the irreducible polynomial $P(u,v)$ to $L(r,\psi)$ yields a polynomial in one variable of degree at most $\delta$. Therefore, either $N_i(r,\psi)\le \delta$, or $P$ vanishes identically on $L(r,\psi)$, in which case $L(r,\psi)\subset F$. Since $P$ is irreducible, the latter can occur for at most one pair $(r,\psi)$, otherwise, $P$ would be divisible by two distinct linear factors. Consequently, $N_i(r,\psi)$ is infinite only on a set of measure zero in the $(r,\psi)$-plane, and thus
 $ 
        N_i(r,\psi)\le \delta
 $
for almost every $(r,\psi)$. Substituting this bound into \eqref{eq:crofton-restricted}, we obtain
$$
{\rm length}(\Gamma_i)
\le
2\sqrt2\pi\delta.
$$
Summing over $1\le i\le \alpha$ gives
$$
\sum_{i=1}^{\alpha}{\rm length}(\Gamma_i)
\le
2\sqrt2 \pi\alpha\delta.
$$
Finally, using the bound for $\alpha$ obtained in part $(ii)$ completes the proof.
\end{proof}
With Lemma \ref{lemma: alpha} at hand, we proceed to prove Theorem \ref{thm:sign change}.
Let
$$A := \min_{-2\le u,v\le 2}P(u, v)    \quad {\rm and} \quad B := \max_{-2\le u,v\le 2}P(u, v),$$ 
so that without loss of generality, we may assume that  $I\subset [A, B]$. Let $I=[t_1,t_2]$ (the proofs in the cases of open and semi-open intervals are similar).
By the Hecke relation \eqref{eq:hecke}, we have
$$P(a(p),a'(p))=P(u,v), \quad {\rm for}~p \nmid NN'.$$ Consequently,
\begin{equation}\label{eq:proof2.1}
\{p\le x :P(a(p),a'(p))\in I\}=\{p\le x : (a(p),a'(p))\in  E_{I,P}\},
\end{equation}
where 
$$
E_{I,P}=\{(u,v)\in [-2,2]^2:t_1\le  P(u,v)\le t_2\}.$$ 
Since the boundary of $E_{I,P}$  is contained in the union of the real algebraic curves
$$u=\pm 2,\quad v=\pm 2, \quad P(u,v)=t_1,\quad P(u,v)=t_2,$$
inside $[-2,2]^2$, 
therefore, applying Lemma \ref{lemma: alpha} to these curves, we see that 
$$
    \partial E_{I,P} = \Gamma_1\cup \cdots\cup\Gamma_{\alpha}
$$
where each $\Gamma_i$ is a continuous curve and more importantly, the zero locus of a polynomial.
By Lemma \ref{lem:alg}, $E_{I,P}$ satisfies Hypothesis \ref{hypothesis} for $\beta=\delta$, and thus an application of \Cref{thm:main2} gives
\begin{equation}\label{eq: ejst E I P}
\#\{p\le x : (a(p),a'(p))\in  E_{I,P}\} = \mu_{\rm JST}(E_{I,P})\pi(x) + O\left( L\alpha\beta\dfrac{\pi(x)}{\mathcal M(x)^{1/2}}\right).
\end{equation}
Further, it follows from Lemma \ref{lemma: alpha} that 
    $$
              \alpha = O(\delta^3) \quad \text{and} \quad  L = {\rm length}(\partial E) = O(\delta^4),    
    $$  
    and substituting these in \eqref{eq: ejst E I P} and
    combining  with \eqref{eq:proof2.1} concludes the proof. 
\section{Proof of \Cref{thm:simultaneous sign change of symmetric power coefficients}}
Define the sets 
\begin{align*}
E_{+}&=\{(u,v)\in [-2,2]^2: U_m(u/2)U_n(v/2)>0\}~{\rm and}\\
    E_{-}& =\{(u,v)\in [-2,2]^2: U_m(u/2)U_n(v/2)<0\}.
    \end{align*}
Applying  \Cref{thm:sign change} for the polynomial $P(u,v)=U_m(u/2)U_n(v/2)$ and for the intervals $(0,\infty)$ and  $(-\infty, 0)$, we obtain
\begin{align}\label{eq:2.3}
    \#\{p\le x :  a(p^m)a'(p^n)>0\}&= \mu_{\rm JST}(E_+){\pi(x)} + O\left( \dfrac{\pi(x)}{\mathcal M(x)^{1/2}}\right) ~ {\rm and}\\
    \#\{p\le x :  a(p^m)a'(p^n)<0\}&= \mu_{\rm JST}(E_-){\pi(x)} + O\left( \dfrac{\pi(x)}{\mathcal M(x)^{1/2}}\right), \notag
\end{align}
respectively.
Since $\{(u,v) \in [-2,2]^2: U_m(u/2)U_n(v/2)=0\}$ has joint Sato--Tate measure zero, we have
 $$\mu_{\rm JST}(E_-)=1-\mu_{\rm JST}(E_{+}).$$ 
 Thus, to prove the theorem, it suffices to compute $\mu_{\mathrm{JST}}(E_{+})$.
 Making the change of variables
$$
u=2\cos\theta,\qquad v=2\cos\phi,\qquad \theta,\phi\in[0,\pi]
,$$ 
we see that
$$d\mu_{\rm JST}=\frac{4}{\pi^2}\sin^2\theta\sin^2\phi \,d\theta d\phi.$$
Therefore, we have
\begin{equation*}\label{eq:measure1}
\mu_{\rm JST}(E_{+}) = \frac{4}{\pi^2}\int \int_{U_m(\cos\theta)U_n(\cos\phi)>0} \sin^2\theta \sin^2\phi\, d\theta d\phi.
\end{equation*}
The region of integration is the disjoint union of the sets where both the factors $U_m(\cos\theta)$ and $U_n(\cos\phi)$ are positive, and both are negative. Consequently,
\begin{equation}\label{eq:measure}
  \mu_{\rm JST}(E_{+}):=  d_{m,n}= d_md_n+(1-d_m)(1-d_n),
\end{equation}
where, for $\ell\ge1$,
$$
d_\ell=\frac{2}{\pi}\int_{0}^{\pi}\mathbf{1}_{\{U_{\ell}(\cos\theta)>0\}}\sin^{2}\theta\,d\theta,
$$
with $\mathbf{1}_{\{U_{\ell}(\cos\theta)>0\}}$ denoting the indicator function of the set $\left\{\theta : U_{\ell}(\cos\theta)>0\right\}$.
By \cite[Theorem~1.1]{jaban}, one has
\begin{equation*}\label{eq:dl}
   d_\ell =
\begin{cases}
\frac12, & \ell \text{ is odd},\\
\frac{\ell+2}{2(\ell+1)} - \frac{1}{2\pi}\tan \left(\frac{\pi}{\ell+1}\right), & \ell \text{ is even.}
\end{cases}
\end{equation*}
Combining it with \eqref{eq:2.3} and \eqref{eq:measure} completes the proof.

\section{Proof of \Cref{cor:simultaneous}}
Consider the sets
\begin{align*}
    E_{+} &\coloneqq \{(u,v)\in [-2,2]^2: P(U_m(u/2),U_n(v/2)) > 0\} ~{\rm and} \\
        E_{-} &\coloneqq \{(u,v)\in [-2,2]^2: P(U_m(u/2),U_n(v/2)) < 0 \}.
\end{align*} 
  Hence, in view of  the identity
$$P(a(p^m),a'(p^n))=P(U_m(a(p)/2), U_n(a'(p)/2)),\quad{\rm for}~p \nmid NN',$$ 
  applying Theorem \ref{thm:sign change}  for the polynomial $P(U_m(u/2),U_n(v/2))$ and intervals $(0,\infty)$ and $(-\infty,0)$, respectively, we obtain
\begin{align*}\label{eq:2.5}
    \#\{p\le x :  P(a(p^m),a'(p^n))>0\}&= \mu_{\rm JST}(E_+){\pi(x)} + O\left( \dfrac{\pi(x)}{\mathcal M(x)^{1/2}}\right)~ {\rm and}\\
    \#\{p\le x :  P(a(p^m),a'(p^n))<0\}&= \mu_{\rm JST}(E_-){\pi(x)} + O\left( \dfrac{\pi(x)}{\mathcal M(x)^{1/2}}\right). \notag
\end{align*}
Since the sets $E_{+}$ and $E_{-}$ are complementary  up to a set of measure zero, we have  
\begin{equation*}
    \mu_{\rm JST}(E_{+})+ \mu_{\rm JST}(E_{-})=1.
\end{equation*}
Therefore, to complete the proof, it suffices to show that \begin{equation}\label{eq:equal}
    \mu_{\rm JST}(E_{+})=\mu_{\rm JST}(E_{-})=\frac{1}{2}.
\end{equation}
Firstly, let $m$ or $n$ be odd. Suppose that $
   P(su,tv)=-P(u,v),$  ${\text{for some }} s\in\{(\pm 1)^m\},\, t\in\{(\pm 1)^n\}.
   $
  For such choices of $s$ and $t$, one can always choose  $\alpha,\beta\in\{\pm1\}$ such that 
$$\alpha^m=s, ~~~~\beta^n=t.$$
The symmetry condition on $P$ and the identity,
$
U_\ell(-u)=(-1)^\ell U_\ell(u){\text{ for }}\ell\ge0, \,u\in\R,$
give
\begin{equation}\label{eq:simultaneousid1}
    P\big(U_m(\alpha u/2),\,U_n(\beta v/2)\big)
= -P\big(U_m(u/2),\,U_n(v/2)\big).
\end{equation}
Next, we define an involution map $\sigma:[-2,2]^2\to[-2,2]^2$ given by
$$
\sigma(u,v):= (\alpha u, \beta v).
$$
This map simply changes the signs of $u$ and $v$ (depending on $\alpha$, $\beta$), hence the Jacobian of 
$\sigma $ has the absolute value $1$. Thus, $\sigma$ preserves the measure $\mu_{\rm JST}$. In particular, we have
\begin{equation}\label{eq:simultaneousid123}
\mu_{\rm JST}(E_{+})= \mu_{\rm JST}(\sigma(E_{+})).
\end{equation}
From the definitions of $\sigma$ and the set $E_{+}$, we can write
\begin{align*}
    \sigma(E_{+})&=\{(\alpha u,\beta v)\in [-2,2]^2: P(U_m(u/2),U_n(v/2)) > 0\}\\
    &=\{(u,v)\in [-2,2]^2: P(U_m(\alpha u/2),U_n(\beta v/2)) > 0\}.
\end{align*}
Using \eqref{eq:simultaneousid1}, we have
$$
 \sigma(E_{+})=\{(u,v)\in [-2,2]^2: P(U_m(u/2),U_n(v/2)) < 0\} =E_{-}.
$$
Combining this with \eqref{eq:simultaneousid123} completes the proof of  \eqref{eq:equal} in this case. 

Next, assume that $m=n$ and $P(u,v)=-P(v,u)$. 
Let $\sigma:[-2,2]^2\rightarrow [-2,2]^2$ be the map defined by 
$\sigma(u,v) =(v,u).$ Using similar arguments as before, one can easily show that $\sigma$ is a measure preserving involution map and 
$$\sigma(E_{+})=E_{-}$$
from which \eqref{eq:equal} follows.
\section{Proof of \Cref{thm:first sign change}}
From \Cref{thm:simultaneous sign change of symmetric power coefficients}, there exists a real number $d> 0$ (depending on $m$ and  $n$) such that
$$\#\{p \leq x: a(p^m)a'(p^n)<0 \} \geq (1-d_{m,n})\pi(x) - d \frac{\pi(x)}{\mathcal{M}(x)^{1/2}},
$$
where $\mathcal{M}(x)$ and $d_{m,n}$ are given by \eqref{eq:mx} and \eqref{eq:dmn}, respectively. To get an upper bound for $p_{f,f',m,n}$, it suffices to get an $x>0$ for which the right-hand side in the last inequality is strictly positive. Thus, we now determine a value of $x$, such that 
\begin{equation*}
(1-d_{m,n})> d\frac{1}{\mathcal{M}(x)^{1/2}}.
\end{equation*}
By the definition of $\mathcal M(x)$, and letting $A = {d}/{(1-d_{m,n})}$, we obtain
\begin{equation}\label{eq: proof of first sign change}
         (\log x)^{1/2} > A^4 \log(kk'NN'\log x).
\end{equation}
Putting $y=\log x$ gives
$$
\sqrt y > A^4(\log (kk'NN') + \log y).
$$
Thus, we aim to find $y$ for which
$$
\sqrt y > 2A^4\log (kk'NN') \quad\text{and}\quad  \sqrt y > 2A^4\log y.
$$
The first inequality holds whenever $$y\ge 4A^8(\log (kk'NN'))^2,$$
whereas, the second inequality holds true for all $y\ge C_1$, for some constant $C_1$ (depending on $A$, that is on $m$ and $n$) as
$\sqrt y/\log y \to\infty$.
Hence, for 
$$ x \ge \exp(C(\log (kk'NN'))^2)$$ with a suitable constant $C>0$, inequality \eqref{eq: proof of first sign change} follows, which completes the proof of \Cref{thm:first sign change}.

\section{Numerical illustration of simultaneous sign and dominating behaviour.}
Take $k=4$, $k'=6$, $N=5$, and $N'=6$. Consider the newforms
\begin{align*}
f(z) &= \sum_{n\ge 1} a(n)n^{3/2} q^n 
= q - 4q^2 + 2q^3 + 8q^4 - 5q^5 - 8 q^6 + 6q^7 - 23q^9 + O(q^{10}), \\
f'(z) &= \sum_{n\ge 1} a'(n)n^{5/2} q^n 
= q + 4q^2 - 9q^3 + 16q^4 - 66q^5 - 36q^6 + 176q^7 + 64q^8 + 81q^9 + O(q^{10}),
\end{align*}
which are the unique normalised non-CM newforms in $S_4(5)$ and $S_6(6)$, respectively.

Figure \ref{fig: same sign} approximates the frequency of primes $p\le x$ for which the products $a(p)a'(p)$, $a(p)a'(p^2)$  and $a(p^2)a'(p^2)$ are positive, where we have taken $x\le100000$. For each $x$ (shown on the $x$-axis), we plot on the $y$-axis the fraction 
$$\#\{p\le x : a(p^m)a'(p^n)>0\}/\pi(x)$$
which approximates the density $d_{m,n}$, where $d_{m,n}$ is given by \eqref{eq:dmn}.
  When $(m,n)=(1,1)$ and $(m,n)=(1,2)$, evidently, this proportion quickly stabilises near $1/2$, indicating that positive and negative signs occur with roughly the same frequency. In contrast, when $(m,n)=(2,2)$, the proportion settles at a value strictly larger than $1/2$, revealing a clear bias toward positivity ($d_{2,2}=0.534\dots$). These behaviours are in complete agreement with \Cref{thm:simultaneous sign change of symmetric power coefficients}, which asserts that no sign bias appears when at least one of the symmetric powers is odd, while a positive bias emerges when both are even.
%\begin{comment}
 \begin{figure}[h]
\includegraphics[width=0.6\linewidth]{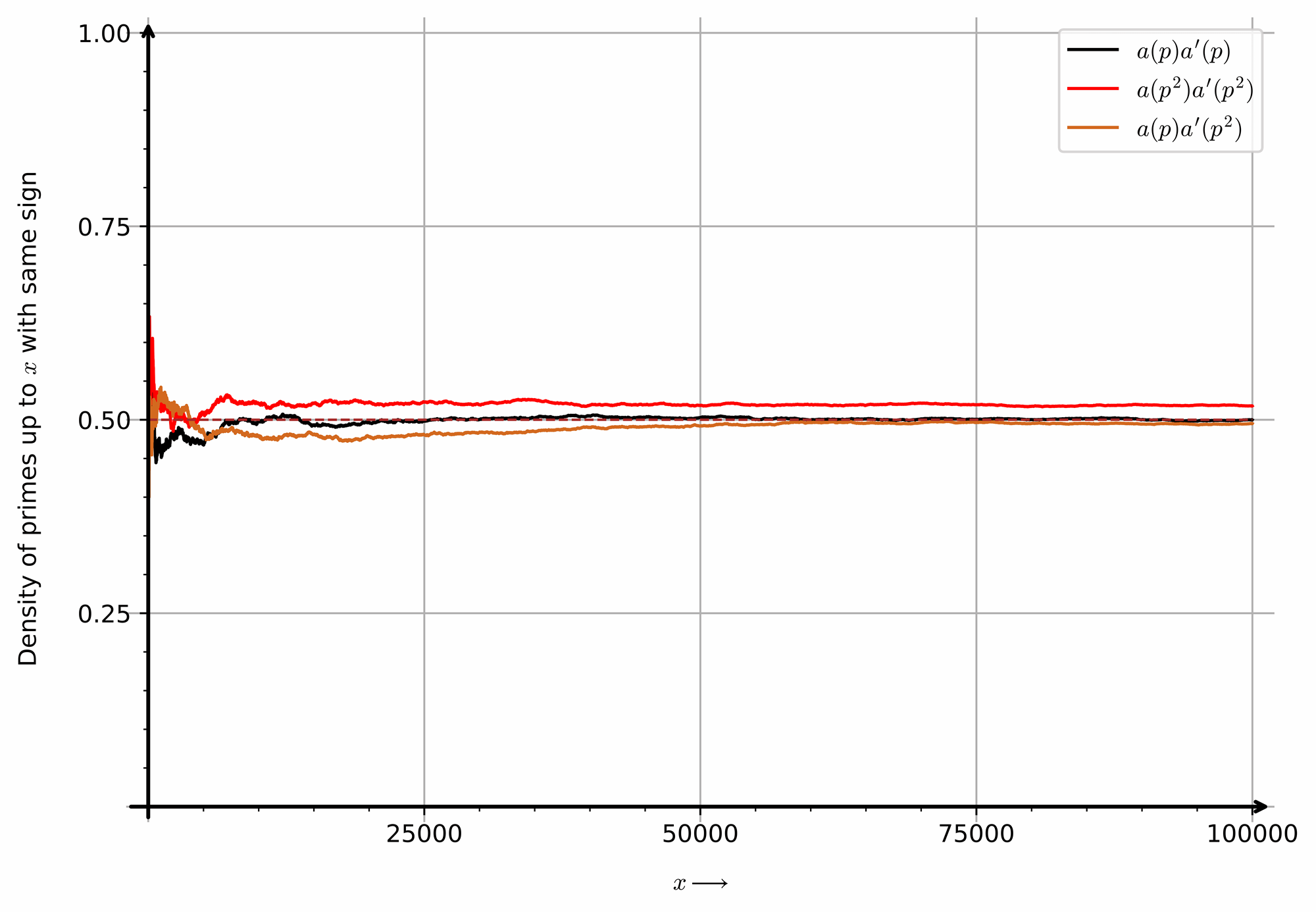}
\caption{Density of primes $p\le x$ such that $a(p)a'(p)$, $a(p^2)a'(p^2)$, and $a(p)a'(p^2)$ are positive,}
    \label{fig: same sign}
\end{figure}
%\end{comment}

Figure \ref{fig: dominating coefficients} demonstrates the proportion of primes $p\le x$ that satisfy $a(p) < a'(p^3)$, $a(p^2)< a'(p^2)$, and $a(p^2)<a'(p^3)$, where we have taken $x\le100000$. For each $x$ (shown on the $x$-axis), we plot on the $y$-axis the fraction 
$$\#\{p\le x : a(p^m)<a'(p^n)\}/\pi(x).$$ 
For the choices $(m,n)= (1,3)$ and $(m,n)= (2,2)$, the densities appear rapidly converging towards $1/2$. This indicates that the Fourier coefficients dominate one another with asymptotically equal frequency. On the other hand, for $(m,n)= (2,3)$, the numerical evidence shows that the density exceeds $1/2$, suggesting that when $m$ and $n$ are of opposite parity, the Fourier coefficients of one form dominate those of the other more frequently which is in complete agreement with Corollary \ref{thm:compare}.
%\begin{comment}
\begin{figure}[h]
    \includegraphics[width=0.6\linewidth]{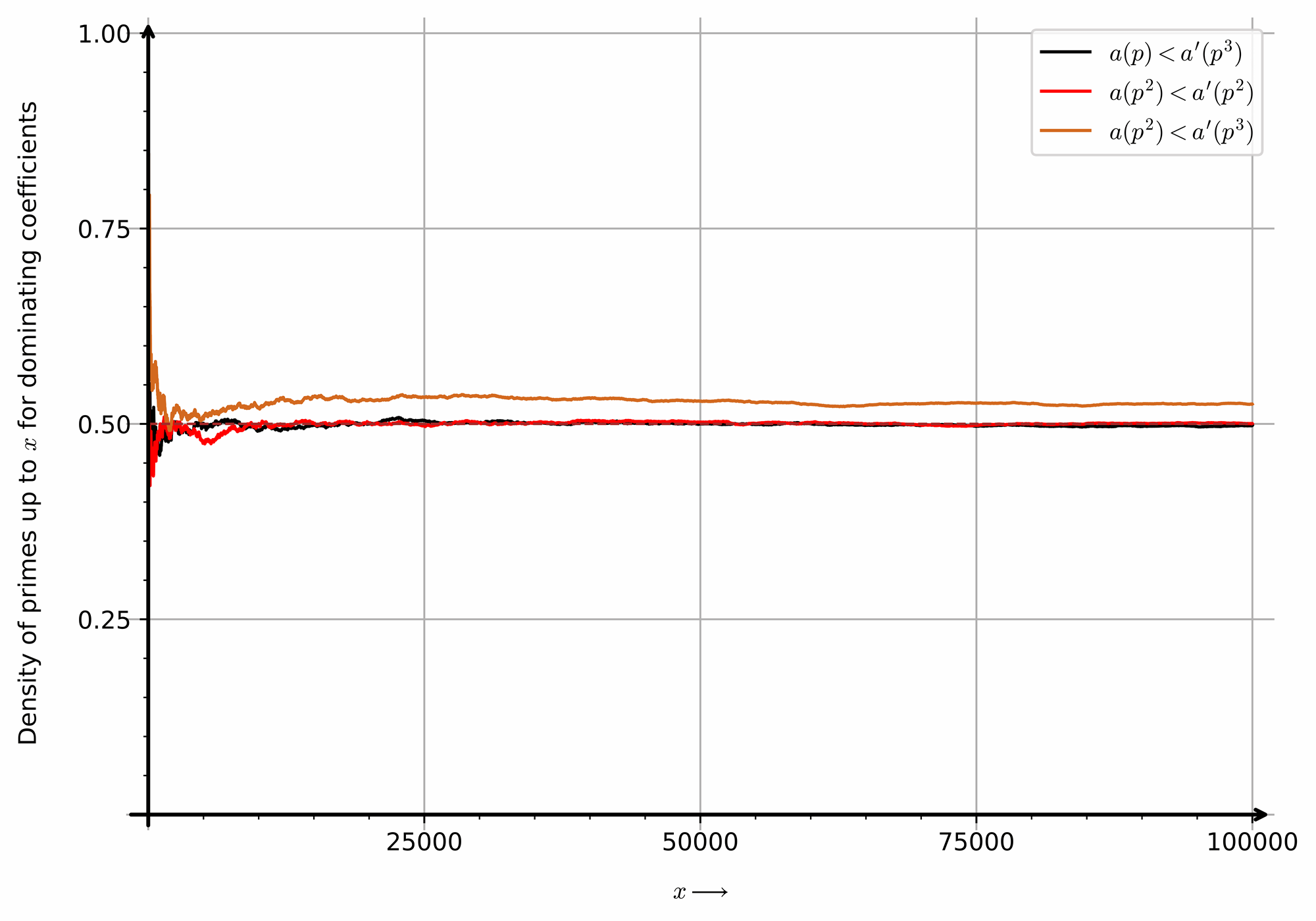}
    \caption{Density of primes $p\le x$ such that $a(p) <a'(p^3)$, $a(p^2)< a'(p^2)$, and $a(p^2)<a'(p^3).$}
    \label{fig: dominating coefficients}
\end{figure}
%\end{comment}
\section*{Acknowledgement}
The authors are grateful to Prof. J. Thorner for his valuable suggestions during the early stages of this work and for drawing their attention to \cite{thorner2}. The open-source mathematical software SageMath was used to generate the plots appearing in this paper. AK was supported by ANRF under the Research Grant ANRF/ARGM/2025/000494/MTR. MK was supported by SEED Grant SGT-100115.
    \bibliography{reference}
	\bibliographystyle{alpha}

\end{document}